\documentclass[10pt, twoside, reqno]{amsart}

\usepackage{xcolor}

\usepackage[english]{babel}

\usepackage{amsmath,amssymb,amsthm}

\usepackage{enumitem}

\usepackage{spverbatim}

\newcommand{\N}{{\mathbb N}}

\newcommand{\Q}{{\mathbb Q}}
\newcommand{\R}{{\mathbb R}}

\newcommand{\F}{{\mathbb F}}

\newcommand{\Sym}{{\mathcal S}}
\newcommand{\T}{{\mathcal T}}

\newcommand{\besep}{\be_{\rm sep}}
\newcommand{\sisep}{\si}

\DeclareMathOperator{\tr}{tr}
\DeclareMathOperator{\GL}{GL}

\newenvironment{eq}{\begin{equation}}{\end{equation}}

\newcommand{\upto}{,\ldots ,}
\renewcommand{\Ref}[1]{(\ref{#1})}

\newcommand{\sumint}[3]{#1\ast {\rm int}_{#3}(#2)}
\renewcommand{\int}[2]{{\rm int}_{#2}(#1)}

\newcommand{\si}{\sigma}
\newcommand{\al}{\alpha}
\newcommand{\be}{\beta}
\newcommand{\ga}{\gamma}

\newcommand{\De}{\Delta}

\newcommand{\ov}[1]{\overline{#1}}
\newcommand{\un}[1]{{\underline{#1}} }

\newcommand{\Ker}{{\mathop{\rm{Ker }}}}

\newtheorem{lem}{Lemma}[section]
\newtheorem{thm}[lem]{Theorem}

\newtheorem{cor}[lem]{Corollary}

\theoremstyle{definition}

\newtheorem{example}[lem]{Example}
\newtheorem{remark}[lem]{Remark}

\numberwithin{equation}{section}

\usepackage[pdfpagemode=None,
           	pdftitle={Separating invariants for multisymmetric polynomials},
           	pdfauthor={Artem Lopatin, Fabian Reimers}]{hyperref}

\begin{document}

\title{Separating invariants over finite fields}

\author{Gregor Kemper, Artem Lopatin, Fabian Reimers}

\address{State University of Campinas, 651 Sergio Buarque de Holanda, 13083-859 Campinas, SP, Brazil}

\email{artem\underline{\;\;}lopatin@yahoo.com}

\address{Technische Universit\"at M\"unchen, Zentrum Mathematik - M11, 
Boltzmannstr.~3, 85748 Garching, Germany}

\email{kemper@ma.tum.de}
\email{reimers@ma.tum.de}




\thanks{Acknowledgments. The second author was supported by FAPESP 2019/10821-8. We are grateful for this support.}

\begin{abstract}

We determine the minimal number of separating invariants for the invariant ring of a matrix group $G \leq \GL_n(\F_q)$ over the finite field $\F_q$. We show that this minimal number can be obtained with invariants of degree at most $|G|n(q-1)$. In the non-modular case this construction can be improved to give invariants of degree at most $n(q-1)$.

As examples we study separating invariants over the field $\F_2$ for two important representations of the symmetric group. 
\end{abstract}

\maketitle

\section*{Introduction}\label{section_intro}
For many properties of an invariant ring, e.g. the number of elements and the degrees in a minimal homogeneous generating system, it is possible to assume that the ground field is algebraically closed. This is not true for the notion of separating invariants and the aim of this paper is to study separating invariants for invariant rings over finite fields. 

We consider an $n$-dimensional vector space $V$ over a field $\F$.
The coordinate ring $\F[V]  = \F[ x_1 \upto x_n]$ of $V$ is isomorphic to the symmetric algebra $S(V^{\ast})$ over the dual space $V^{\ast}$, with $x_1 \upto x_n$ a basis of $V^\ast$. 
Let $G$ be a finite subgroup of $\GL(V) \cong \GL_n(\F)$. The algebra $\F[V]$ becomes a $G$-module with
$(g \cdot f)(v) = f(g^{-1}\cdot v)$ for all $f \in V^\ast$ and $v \in V$.
The ring of invariants 
$$\F[V]^G=\{f\in \F[V] \mid g\cdot f=f \text{ for all }g\in G\}$$ 
is a finitely generated $\F$-algebra.

In this paper we are interested in separating invariants.
Given a subset $S$ of $\F[V]^G$, we say that elements $u,v$ of $V$ {\it can be separated by $S$} if there exists an invariant $f\in S$ with $f(u)\neq f(v)$. A subset $S\subset \F[V]^G$ of the invariant ring is called {\it separating} if any $u, v \in V$ that can be separated by $\F[V]^G$ can also be separated by $S$. Since $G$ is finite, elements $u, v$ that are not in the same $G$-orbit can always be separated by $\F[V]^G$ (see~\cite[Section 2.4]{derksen2002computationalv2}). There exists a general construction of a finite separating set for $\F[V]^G$ consisting of invariants of degree at most $|G|$ (see \cite[Theorem 3.12.1]{derksen2002computationalv2}), but this construction does not yield a separating set of minimal size. Let $\besep(\F[V]^G)$  denote the minimal integer $\besep$ such that the set of all homogeneous invariants of degree  less or equal to $\besep$ is separating for $\F[V]^G$. Explicit minimal separating sets or upper bounds on $\besep$  have been calculated for various groups, see e.g. 
\cite{Cavalcante_Lopatin_1}, \cite{derksen2018}, \cite{domokos2017}, \cite{domokos20Add}, \cite{domokos20}, \cite{dufresne2014}, \cite{elmer2014}, \cite{Ferreira_Lopatin},  \cite{kaygorodov2018}, \cite{kohls2013}, \cite{reimers2020}.

In all these papers, $\F$ is assumed to be an infinite field, usually even algebraically closed.  In~\cite{dufresne2009separating}, Dufresne coined the term {\em geometrically separating} for invariants that remain separating when one passes to the vector space $\ov{V} = V \otimes_{\F} \ov{\F}$. For generating invariants, passing to the algebraic closure does not make any difference. But for separating invariants it does, so geometrically separating and separating are different concepts. For example, the single polynomial $x^3$ is a separating invariant for the trivial group acting on $V = \R^1$, but it is clearly not geometrically separating. Moreover,  every geometrically separating set must contain at least $n$ elements and the existence of separating sets of size $n$ was studied in \cite{dufresne2009separating} and \cite{reimers2018}. But as we will see in this paper, over finite fields there often exists a separating set of size less than~$n$.  Geometric separating sets over finite fields were considered  for some groups in a recent paper~\cite{Chen_Shank_Wehlau_21}.  

This paper may be the first one in the literature that explicitly deals with (non-geometrically) separating invariants over finite fields. This is because usually one uses separating invariants in geometric situations or to deal with problems in characteristic~$0$. However, the finite field case may be applicable to problems in discrete mathematics, and the third section of this paper, dealing with the isomorphism problem of simple graphs, points to such an application. So from now on let $\F = \F_q$ be a field with $q$ elements and throughout this paper let $k$ denote the number of $G$-orbits in $V$.

In Theorem~\ref{theo1} in Section~\ref{section_res} we prove that the minimal number of separating invariants is 
\begin{equation*}
\ga=\ga(q,k) :=\lceil\log_q(k)\rceil,
\end{equation*}
which is bounded above by~$n$ but very often smaller than~$n$.  We give an explicit construction of a minimal separating set of $\ga$ invariants of degree at most $|G|n(q-1)$.

Section~\ref{section_res2} is about improvements of Theorem~\ref{theo1}. In Theorem~\ref{theo2} we construct a minimal separating set of $\ga$ invariants of degree $n(q-1)$ in the non-modular case (i.e., when $|G|$ is invertible in $\F$) and in the case when $G$ consists entirely of monomial matrices (i.e., matrices that have exactly one nonzero entry in each row and in each column). The latter includes the case of permutation groups.

The last two sections  of the paper study separating invariants for two important representations of the symmetric group $\Sym_n$. The use of the letter $n$ now differs from before, as in Section~\ref{section_graphs} we consider the $\Sym_n$-representation of dimension $\binom{n}{2}$ defined by the action of $\Sym_n$ on subsets of size two of $\{ 1 \upto n \}$. Because of its importance to the graph isomorphism problem, the corresponding invariant ring has been studied over a field of characteristic zero in the past where generators are known for $n \leq 5$. But for solving the isomorphism problem for simple graphs, it is enough to consider separating invariants over the field $\F_2$. In Theorems~\ref{theoGraphs} and~\ref{theoGraphs2} we give minimal separating invariants over the field $\F_2$ for the cases $n = 4$ and $n = 5$.

In Section~\ref{section_Sn} we consider multisymmetric polynomials over a finite field. They are the elements of the invariant ring $\F[V^m]^{\Sym_n}$ where $\Sym_n$ acts on $V = \F^n$ by permuting the coordinates and then diagonally on the direct sum $V^m = V \oplus \ldots \oplus V$. In the case $\F = \F_2$ we give an explicit separating set for $\F[V^m]^{\Sym_n}$ which is minimal w.r.t. inclusion in Theorem~\ref{theo_F2}. We also compute that $\besep$ for $\F[V^m]^{\Sym_n}$ is equal to $2^{\lfloor  \log_2(n)\rfloor}$ (see Corollary~\ref{cor}).

\section{Minimal separating invariants}\label{section_res}

Recall that we are working over the finite field $\F = \F_q$. For any vector $w\in V$ with components $w_1 \upto w_n \in \F$ consider the following polynomial in $\F[V]$: 
$$f_w = (-1)^n \prod\limits_{a\in\F \setminus \{w_1 \}}(x_1 - a) \;\;\cdots \prod\limits_{a\in\F\setminus \{w_n\}}(x_n - a).$$
Note that $f_w(w)=(-1)^n \left(\prod\limits_{a\in\F^{\ast}} a \right)^{\!\!n} =  1$ and for each $v\in V$ we have 
$$f_w(v)=\left\{
\begin{array}{cl}
1,& v=w \\
0,& \text{otherwise} \\
\end{array}
\right..$$
Let $k$ denote the number of $G$-orbits in $V$ and decompose $V$ as a union of orbits: $V=W_1 \sqcup \cdots \sqcup W_k$. For every $1\leq j\leq k$ we define
\begin{equation}\label{eq_fj}
f_j=\sum\limits_{w\in W_j} f_w.
\end{equation}
Then for each $v\in V$ we have
\begin{equation}\label{eq_f}
f_j(v)=\left\{
\begin{array}{cl}
1,& v\in W_j \\
0,& \text{otherwise} \\
\end{array}
\right..
\end{equation}
Note that the $f_j$'s are not necessarily invariants. To remedy this, for every $1\leq j\leq k$ introduce the following elements of $\F[V]^G$:
$$n_j=\prod\limits_{g\in G} (g\cdot f_j).$$
As for the  $f_j$, for each $v\in V$ we have 
\begin{equation}\label{eq_n}
n_j(v)=\left\{
\begin{array}{cl}
1,& v\in W_j \\
0,& \text{otherwise} \\
\end{array}
\right..
\end{equation}
Hence $n_1 \upto n_k$ form a separating set for $\F[V]^G$.

Since for
\begin{equation}\label{eqGamma}
\ga=\ga(q,k) :=\lceil\log_q(k)\rceil,
\end{equation}
we have  $q^{\ga-1}<k\leq q^{\ga}$, so there exist~$k$ pairwise different vectors in $\F^\gamma$, which we may put together in a matrix $(a_{ij})$ over $\F$ of size $(\gamma \times k)$.

\begin{thm}\label{theo1} 
\begin{enumerate}
\item[1.] Every separating set for $\F[V]^G$ contains at least $\ga$ elements. 

\item[2.] Let $(a_{ij}) \in \F^{\gamma \times k}$ be a matrix whose columns are pairwise different and define $$t_i =  a_{i1} n_1 + \cdots + a_{ik} n_k \quad (1\leq i\leq \ga).$$
Then $t_1 \upto t_{\ga}$ form a minimal separating set for $\F[V]^G$ consisting of (non-homogeneous) elements of degree less or equal to $|G|n(q-1)$.
\end{enumerate}
\end{thm}

\begin{proof} Consider some representatives of $G$-orbits: $w_1\in W_1 \upto w_k\in W_k$.

\medskip
\noindent{\bf 1.} Assume that $l_1 \upto l_r$ is a separating set for $\F[V]^G$. Then the column vectors \[ \begin{pmatrix}l_1(w_j)\\ \vdots \\ l_r(w_j) \end{pmatrix} \quad \text{ with } 1 \leq j \leq k \] are pairwise different. Hence $k \leq q^r$ and $\gamma \leq r$ follows from (\ref{eqGamma}).

\medskip
\noindent{\bf 2.} Applying formula~(\ref{eq_n}) we obtain that 
$$\begin{pmatrix}t_1(w_j) \\ \vdots \\ t_{\ga}(w_j) \end{pmatrix} = \begin{pmatrix}a_{1j}\\ \vdots \\ a_{\ga j} \end{pmatrix} \quad \text{ for all } 1\leq j\leq k.$$ 
By assumption, these column vectors are pairwise different, hence $t_1,\ldots,t_{\ga}$ are separating. Minimality follows from part 1.
\end{proof}

\section{Minimal separating sets in special cases}\label{section_res2}

Given $f\in \F[V]$ and $1\leq i\leq n$, let $\deg_{x_i}(f)$ denote the degree of $f$ considered as a polynomial in one variable $x_i$. For $d\in \N$ let $\F[V](d)$ denote the subspace of $\F[V]$ that consists of all $f\in \F[V]$ with $\deg_{x_i}(f)\leq d$ for all $1\leq i\leq n$.  Note that $f_j$ as defined in (\ref{eq_fj}) lies in $\F[V](q-1)$.

Let $\mathcal{O}(V)$ denote the algebra of (polynomial) functions on the vector space $V$, i.e., $\mathcal{O}(V)\simeq \F[V] / \Ker(\Psi)$ for the obvious homomorphism
$\Psi: \F[V]\to \mathcal{O}(V)$%
. It is well-known that $\Ker(\Psi)$ is generated by $x_i^q - x_i$ with $1 \leq i \leq n$. 

\begin{lem}\label{lemma1}
Assume that $g\cdot f_j$ belongs to $\F[V](q-1)$ for all $g\in G$ and all $1\leq j \leq k$. Then $f_1 \upto f_k$ are invariants. Hence they form a separating set for $\F[V]^G$.
\end{lem}
\begin{proof} Let $g \in G$ and $1 \leq j \leq k$. We want to show that $g \cdot f_j = f_j$. Since $f_j$ and by assumption also $g \cdot f_j$ lie in $F[V](q-1)$, we know that 
$g\cdot f_j - f_j \in  \F[V](q-1)$. 

Formula (\ref{eq_f}) implies that \[(g\cdot f_j - f_j)(v)=f_j(g^{-1}\cdot v) - f_j(v) = 0.\] Thus $g\cdot f_j - f_j$ belongs to $\Ker(\Psi)$. Using the fact that $x_1^q - x_1 \upto x_n^q - x_n$ is a Gr\"obner basis of the ideal $\Ker(\Psi)$ and $\deg_{x_i}(g \cdot f_j - f_j) \leq q-1$ for all $i$, we can see that $g \cdot f_j - f_j = 0$. Hence the polynomials $f_j$ are invariants and by formula~(\ref{eq_f}) they separate $G$-orbits on $V$.
\end{proof}

Let $G_n<\GL_n(\F)$ denote the group of {\it monomial matrices}, i.e., the set of all $n\times n$ matrices over $\F$ with exactly one nonzero entry in each row and exactly one nonzero entry in each column. This group is isomorphic to a semidirect product $G_n \simeq \Sym_n \ltimes (\F^{\ast})^n$, where $\Sym_n$ is the symmetric group on $n$ letters and $\F^{\ast}$ is the multiplicative group of the field $\F$.  Hence we may write an element $g\in G_n$ uniquely as $g=(\si; \al_1,\ldots,\al_n)$ with a permutation $\si \in \Sym_n$ and elements $\al_i \in \F^\ast$ such that we have 
\begin{equation}\label{eq_Gn_action}
g \cdot x_i= \al_i x_{\si(i)} \quad \text{(for all }1\leq i\leq n). 
\end{equation}

\begin{lem}\label{lemma_Gn} The group $G <\GL_n(\F)$ is a subgroup of the group $G_n$ of monomial matrices if and only if 
$g\cdot f_j$ belongs to $\F[V](q-1)$ for all $g\in G$ and all $1\leq j \leq k$.
\end{lem}
\begin{proof} 
Assume that $G<G_n$ and take $g=(\si; \al_1,\ldots,\al_n)\in G$ and $w\in V$. Formula~(\ref{eq_Gn_action}) implies that 
$$\deg_{x_i}(g\cdot f_w) = \sum\limits_{r=1}^n \sum\limits_{a\in \F\setminus \{w_r\}} 
\deg_{x_i}(\al_r x_{\si(r)} - a) = q-1.$$
Therefore, $g\cdot f_j$ belongs to $\F[V](q-1)$ for all $j$. 

Now assume that $g\cdot f_j$ belongs to $\F[V](q-1)$ for all $g \in G$ and for $1 \leq j \leq k$. To show that $g$ is a monomial matrix, it is enough to consider this assumption for the polynomial $f_j$ which belongs to the orbit of the zero vector of $V$. We can write $f_j = f_0$ and here we have for every $1 \leq i \leq n$:
$$q-1 \geq \deg_{x_i}(g\cdot f_0) = \sum\limits_{r=1}^n \sum\limits_{a\in \F^{\ast}} 
\deg_{x_i}(g\cdot x_r - a) = (q-1) \sum\limits_{r=1}^n  
\deg_{x_i}(g\cdot x_r).$$
Note that
$$\deg_{x_i}(g\cdot x_r)=\left\{
\begin{array}{cl}
1, & (g^{-1})_{r,i}\neq 0  \\
0, & \text{otherwise} \\
\end{array}
\right. .
$$ 
So the number of nonzero elements in any column of $g^{-1}$ is less or equal to $1$. It follows that $g$ is a monomial matrix. 
\end{proof}

Lemmas~\ref{lemma1},~\ref{lemma_Gn} imply the following statement:

\begin{cor}\label{cor_Gn} If $G<G_n$, then $f_1,\ldots,f_k$ form a separating set for $\F[V]^G$.
\end{cor}

\begin{lem}\label{lemma2}
Assume that $|G|$ is invertible in $\F$. Then
$$h_j=\frac{1}{|G|}\sum\limits_{g\in G} g\cdot f_j \qquad (1\leq j\leq k)$$
form a separating set for $\F[V]^G$.
\end{lem}
\begin{proof}
Obviously, $h_1,\ldots,h_k$ belong to $\F[V]^G$ and for each $v\in V$ we have 
\begin{equation}\label{eq_h}
h_j(v)=\left\{
\begin{array}{cl}
1,& v\in W_j \\
0,& \text{otherwise} \\
\end{array}
\right..
\end{equation} 
Thus $h_1,\ldots, h_k$ separate $G$-orbits on $V$.  
\end{proof}

\begin{thm}\label{theo2} 
Let $(a_{ij}) \in \F^{\gamma \times k}$ be a matrix such that all columns are pairwise different. Then in the following two cases we can construct minimal separating sets $t'_1 \upto t'_{\ga}$ and  $t''_1 \upto t''_{\ga}$, respectively, consisting of elements of degree less or equal to $n(q-1)$: 
\begin{enumerate}
\item[(a)] if $G$ is a subgroup of the group $G_n$ of monomial matrices, then define
$$t'_i =  a_{i1} f_1 + \cdots + a_{ik} f_k \quad (1\leq i\leq \ga);$$

\item[(b)] if $|G|$ is invertible in $\F$, then define
$$t''_i =  a_{i1} h_1 + \cdots + a_{ik} h_k \quad (1\leq i\leq \ga),$$
where $h_1,\ldots,h_k$ were defined in Lemma~\ref{lemma2}.
\end{enumerate}

\end{thm}
\begin{proof} Corollary~\ref{cor_Gn} and Lemma~\ref{lemma2} imply that $t'_i, t''_i\in \F[V]^G$.  Using formulas~(\ref{eq_f}) and~(\ref{eq_h}) instead of formula~(\ref{eq_n}) in the proof of part~2 of Theorem~\ref{theo1} concludes the proof of Theorem~\ref{theo2}.
\end{proof}

\section{Application to simple graphs}\label{section_graphs}

This section deals with an $\Sym_n$-representation over $\F_2$ that is motivated by the ``graph isomorphism problem'' and shows that separating invariants over finite fields are meaningful to study. The symmetric group $\Sym_n$ acts on the set of subsets of size two of $\{ 1 \upto n \}$ by
\[ \sigma \cdot \{i,j \} := \{\sigma(i),\sigma(j) \}.\]
For any field $\F$ this defines a representation $V = \F^m$ of $\Sym_n$ of degree $m := \binom{n}{2}$. The points in $V$ correspond to graphs on $n$ vertices, where the edges are weighted with an element of $\F$. The coordinate ring of $V$ is a polynomial ring in $m$ variables
\[ \F[V] = \F[x_{ij} \mid 1 \leq i < j \leq n].\]
We consider the case $n = 4$. In \cite{aslaksen1996invariants} and \cite{thiery2000algebraic} it was shown that over a field $\F$ of characteristic zero the invariant ring $\F[V]^{\Sym_4}$ is minimally generated by 9 homogeneous invariants. Using the ad-hoc notation $o(f)$ for the sum over the $\Sym_n$-orbit of $f \in \F[V]$ these generators are the following orbit sums of square-free monomials:
\begin{eqnarray*}
f_1 := o(x_{12}), \,\,\, & f_2 := o(x_{12}x_{13}), \,\,\, & f_3 := o(x_{12}x_{34}),\\ 
f_4 := o(x_{12}x_{13}x_{14}), \,\,\, & f_5 := o(x_{12}x_{13}x_{23}), \,\,\, & f_6 := o(x_{12}x_{13}x_{24}),\\ 
f_7 := o(x_{12}x_{13}x_{14}x_{24}), \,\,\, & f_8 := o(x_{12}x_{13}x_{24}x_{34}),  \,\,\, &f_9 := o(x_{12}x_{13}x_{14}x_{23}x_{24}).
\end{eqnarray*}
The case $\F = \F_2$ is also very interesting, as here the points in $V$ correspond to simple graphs without weights and the orbits correspond to isomorphism classes of such graphs. A \textsc{magma} \cite{bosma1997magma} computation reveals that over $\F = \F_2$ the 9 invariants from above together with
\[ f_{10} := o(x_{12}x_{13}x_{14}x_{23}x_{24}x_{34}) = x_{12}x_{13}x_{14}x_{23}x_{24}x_{34}\]
form a minimal generating set of the invariant ring.

Note that there are 11 isomorphism classes of simple graphs on $n = 4$ vertices. So by Theorem~\ref{theo1} the minimal size of a separating set for $\F_2[V]^{\Sym_4}$ is 4. In the following theorem we show that such a set cannot be found directly as a subset of the generating set
\begin{equation}\label{eqGeneratorsGraphs} M := \{ f_1 \upto f_{10} \},\end{equation}
but a separating set of size 4 can be easily obtained from it.

\begin{thm}\label{theoGraphs} 
The subsets $S_1 := \{ f_1,f_2,f_3,f_4,f_8 \}$ and $S_2 := \{ f_1,f_2,f_3,f_5,f_8 \}$ are the only minimal (w.r.t. inclusion) separating subsets of $M$. Moreover, $S_3 := \{ f_1, f_2, f_3, f_4 + f_8 \}$ and $S_4 := \{ f_1, f_2, f_3, f_5 + f_8 \}$ are separating sets for $\F_2[V]^{\Sym_4}$ of the smallest possible size.  
In particular, $\besep(\F_2[V]^{\Sym_4})=4$.
\end{thm}

\begin{proof}
Here $V$ is $6$-dimensional and the coordinates are denoted by $x_{12}$, $x_{13}$, $x_{14}$, $x_{23}$, $x_{24}$, $x_{34}$. We take representatives of the orbits in $V$:
\begin{eqnarray*} p_1 := (0,0,0,0,0,0), \,\,\, & p_2 := (1,0,0,0,0,0), \,\,\, & p_3 := (1,1,0,0,0,0),\\
 p_4 := (1,0,0,0,0,1), \,\,\, & p_5 := (1,1,1,0,0,0), \,\,\, & p_6 := (1,1,0,1,0,0),\\
 p_7 := (1,1,0,0,1,0), \,\,\, & p_8 := (1,1,1,1,0,0), \,\,\, & p_9 := (1,1,0,0,1,1),\\
 p_{10} := (1,1,1,1,1,0), \,\,\, & p_{11} := (1,1,1,1,1,1). \,\,\, &
\end{eqnarray*}
Evaluating the 10 generators of $\F_2[V]^{\Sym_4}$ from \eqref{eqGeneratorsGraphs} on the 11 points gives the $10 \times 11$-matrix:
\[ 
\begin{pmatrix}0&1&0&0&1&1&1&0&0&1&0\\
0& 0& 1& 0& 1& 1& 0& 1& 0& 0& 0\\
0& 0& 0& 1& 0& 0& 1& 1& 0& 0& 1\\
0& 0& 0& 0& 1& 0& 0& 1& 0& 0& 0\\
0& 0& 0& 0& 0& 1& 0& 1& 0& 0& 0\\
0& 0& 0& 0& 0& 0& 1& 0& 0& 0& 0\\
0& 0& 0& 0& 0& 0& 0& 1& 0& 0& 0\\
0& 0& 0& 0& 0& 0& 0& 0& 1& 1& 1\\
0& 0& 0& 0& 0& 0& 0& 0& 0& 1& 0\\
0& 0& 0& 0& 0& 0& 0& 0& 0& 0& 1
\end{pmatrix}.
\]
Comparing column $1$ to columns $2$, $3$, $4$ and $9$, respectively, we see that every separating subset of $M$ must contain $f_1,f_2,f_3$ and $f_8$. From columns $5$, $6$ we see that such a set must also contain $f_4$ or $f_5$. Evaluating $S_1$ and $S_2$ on the 11 points gives the submatrices  
\[ 
\begin{pmatrix}0&1&0&0&1&1&1&0&0&1&0\\
0& 0& 1& 0& 1& 1& 0& 1& 0& 0& 0\\
0& 0& 0& 1& 0& 0& 1& 1& 0& 0& 1\\
0& 0& 0& 0& 1& 0& 0& 1& 0& 0& 0\\
0& 0& 0& 0& 0& 0& 0& 0& 1& 1& 1
\end{pmatrix}
\]
and 
\[ 
\begin{pmatrix}0&1&0&0&1&1&1&0&0&1&0\\
0& 0& 1& 0& 1& 1& 0& 1& 0& 0& 0\\
0& 0& 0& 1& 0& 0& 1& 1& 0& 0& 1\\
0& 0& 0& 0& 0& 1& 0& 1& 0& 0& 0\\
0& 0& 0& 0& 0& 0& 0& 0& 1& 1& 1
\end{pmatrix}.
\]
In both matrices the columns are pairwise distinct, hence $S_1$ and $S_2$ are separating sets. If we add the last two rows together in these matrices, then still all columns are pairwise distinct. This shows that $S_3$ and $S_4$ are separating sets. It follows from Theorem~\ref{theo1} that they are of the smallest possible size.
\end{proof}

The next case is $n = 5$. Over a field $\F$ of characteristic zero a minimal generating set for $\F[V]^{\Sym_5}$ consisting of 57 invariants of degree at most 9 is known (see \cite{pouzet2001invariants} and \cite[Section 5.5]{derksen2002computationalv2}). Again we shift the focus to the finite field $\F = \F_2$ in order to deal with simple graphs and to separating instead of generating invariants. There are 34 isomorphism classes of simple graphs with 5 vertices. So by Theorem~\ref{theo1} the minimal size of a separating set for $\F_2[V]^{\Sym_5}$ is 6. Indeed, as in the case $n = 4$ such a minimal separating set can be obtained from certain orbit sums of square-free monomials. We state this result in the following theorem, but refrain from writing down the computations this time.

\begin{thm}\label{theoGraphs2} 
In the case $n = 5$ the following invariants form a separating set for $\F_2[V]^{\Sym_5}$ of the smallest possible size:
\begin{align*}
&g_1 := o(x_{12}), \\
&g_2 := o(x_{12} x_{13}),\\
&g_3 := o(x_{12} x_{34}) + o(x_{12}x_{13}x_{24}x_{34}),\\
&g_4 := o(x_{12} x_{13} x_{14}) + o(x_{12}x_{13}x_{24}x_{34}),\\
&g_5 := o(x_{12} x_{13} x_{23}) + o(x_{12}x_{13}x_{24}x_{34}),\\
&g_6 := o(x_{12} x_{13} x_{14} x_{15}) + o(x_{14}x_{15}x_{24}x_{25}x_{34}x_{35}).\\
\end{align*}
In particular, $\besep(\F_2[V]^{\Sym_5})\leq 6$.
\end{thm}

\begin{proof}
We do not give the proof, as it consists entirely of computations which are very similar but much longer than those in the proof of Theorem~\ref{theoGraphs}. These computations can be done in \textsc{magma} \cite{bosma1997magma}.
\end{proof}

\section{Separating set for multisymmetric polynomials}\label{section_Sn}

In this section we look at the diagonal action of $G=\Sym_n$ on $V^m = V \oplus \cdots \oplus V$ by permuting elements of a fixed basis of $V$. The induced action on the coordinate ring
$$ \F[V^m] = \F[x(j)_i \,|\, 1\leq i \leq n,\, 1\leq j \leq m]$$
is given by $\si\cdot x(j)_i = x(j)_{\si(i)}$ for $\si\in\Sym_n$. Upper bounds on the degrees of elements of minimal generating sets for the algebra of multisymmetric polynomials $\F[V^m]^{\Sym_n}$ were studied by Fleisch\-mann~\cite{fleischmann}, Vaccarino~\cite{vaccarino2005}, Domokos~\cite{domokos2009vector}, and a minimal generating set was explicitly described by Rydh~\cite{rydh2007}.

Given a vector $\un{\al}=(\al_1,\ldots,\al_r)\in \N^r$ we denote $|\un{\al}|=\al_1+\cdots +\al_r$ and $\#\un{\al}=r$. We write $\gcd(\un{\al})$ for the greatest common divisor of $\al_1,\ldots,\al_r$. For any $\un{\al}=(\al_1,\ldots,\al_m)\in \N^m$ and $1\leq i,t\leq n$ we set $x^{\un{\al}}_i=x(1)_i^{\al_1}\cdots x(m)_i^{\al_m}$ and denote by 
$$\si_t(\un{\al})=\sum\limits_{1\leq i_1<\cdots< i_t\leq n} x^{\un{\al}}_{i_1} \cdots x^{\un{\al}}_{i_t} \;\in\; \F[V^m]^{\Sym_n}$$
the \emph{elementary multisymmetric polynomials}. In case $\un{\al}=(0,\ldots,0)$ we have $x^{\un{\al}}_i=1$. For short, we write $\tr(\un{\al})=\si_1(\un{\al})$ for the power sum
\[\tr(\un{\al}) = \si_1(\un{\al})= \sum\limits_{i=1}^n x^{\un{\al}}_i.\]
Over an arbitrary field $\F$ the algebra $\F[V^m]^{\Sym_n}$ is known to be generated by the $m$ elements $\si_n(0,\ldots,0,1,0,\ldots,0)$, where 1 is at position $j$ for $1 \leq j \leq m$, together with the elements $\si_t(\un{\al})$ with $1\leq t\leq n$, $\gcd(\un{\al})=1$ and $\al_j<n/t$ for all $j$ (see Corollary 5.3 of~\cite{domokos2009vector}). There is a natural $\N^m$-grading on $\F[V^m]$, which is preserved by the $\Sym_n$-action. Under this grading $\si_t(\un{\al})$ is homogeneous of multidegree $t\,\un{\al}$.
 
Since $\Sym_n$ is a monomial matrix group, a minimal separating set $S$ for $\F[V^m]^{\Sym_n}$ with the least possible $|S|=\ga$ is given by part~(a) of Theorem~\ref{theo2}. 

In case $m=1$ we write $x_i=x(1)_i$ for the variables and $s_t=\si_t(1)\in \F[V]^{\Sym_n}$ for the elementary symmetric functions in $x_1,\ldots,x_n$. For $i,r\geq 0$ denote by $\xi_r(i)\in\N$ the $r$-th digit of $i$ written as a binary number, i.e., $\xi_r(i)=i_r$, where for $d=\lfloor\log_2(i)\rfloor$ we have $i=(i_d i_{d-1}\ldots i_0)_2$ with $i_d=1$ and $i_{d+1}=i_{d+2}=\cdots=0$. Let us recall Legendre's formula in the partial case of the prime equal to two:

\begin{remark}\label{remark_F2}
Given $i>0$, denote by $\nu(i)\in\N$ the maximal power of two that divides $i$ and we write $\xi(i)\in\N$ for the sum of all digits of $i$ written as a binary number, i.e., $\xi(i)=\sum_{j\geq 0} \xi_j(i)$.  Then $\nu(i!)=i-\xi(i)$.
\end{remark}

Note that for $\F=\F_2$ representatives of $\Sym_n$-orbits on $V$ are $e_i=(1,\ldots,1,0,\ldots,0)$, where $1$ appears $i$ times and $0\leq i\leq n$. 

\begin{lem}\label{lem_F2_binary} Assume $\F=\F_2$. Then
\begin{enumerate}
\item[(a)] $s_{2^r}(e_i)=\xi_r(i)$ for all $r,i\geq 0$, where we consider $\xi_r(i)$ as the element of $\F_2$.

\item[(b)] for every $0\leq a,b\leq n$ with $\De=a-b>0$ invariants  $$\{s_{2^r}\,|\,0\leq r\leq \lfloor \log_2(\De)\rfloor\}$$ 
separate elements $e_{a}$ and $e_{b}$ of $V$.
\end{enumerate}
\end{lem}
\begin{proof} \noindent{\bf (a)}\; If $i<2^r$, then $s_{2^r}(e_i)=0$ and the claim of the lemma holds. 

Assume that $i\geq 2^r$. Then $s_{2^r}(e_i) = \binom{i}{2^r}$.  By Remark~\ref{remark_F2} we have that $\nu(\binom{i}{2^r})= \xi(i-2^r) + 1 - \xi(i)$. If $i=(i_d\ldots i_{r+1}1\,i_{r-1}\ldots i_0)_2$, then $i-2^r=(i_d\ldots i_{r+1}0\,i_{r-1}\ldots i_0)_2$ and $\nu(\binom{i}{2^r})=0$. On the other hand, if 
$$i=(i_d\ldots 1\underbrace{0\ldots0}_{j}0\,i_{r-1}\ldots i_0)_2$$
for some $j\geq0$, then 
$$i - 2^r=(i_d\ldots 0\underbrace{1\ldots1}_{j}1\,i_{r-1}\ldots i_0)_2$$
and $\nu(\binom{i}{2^r})=j+1>0$. Since the characteristic of $\F$ is two, the proof is completed. 

\medskip
\noindent{\bf (b)}\; Assume that $s_{2^r}(e_{a})=s_{2^r}(e_{b})$ for all $0\leq r\leq \lfloor \log_2(\De)\rfloor$. Since $\xi_d(\De)=0$ for all $d>\lfloor \log_2(\De)\rfloor$, part~(a) implies that $\xi_r(a)=\xi_r(b)$ for all $r$, i.e., $\De=0$; a contradiction.
\end{proof}

The next lemma together with Lemma~\ref{lemma_reduction} (see below) allows us to diminish separating sets in some partial cases.

\begin{lem}\label{lem_F2_m1}
If $\F=\F_2$ and $m=1$, then the following set is a minimal separating set for $\F[V]^{\Sym_n}$ containing the least possible number of elements:
$$S_{n,1} = \{s_{2^r}\,|\,0\leq r\leq \lfloor \log_2(n)\rfloor\}.$$
\end{lem}
\begin{proof} By part~(a) of Lemma~\ref{lem_F2_binary}, considering $s_{2^r}(e_i)$ for all $0\leq r\leq \lfloor \log_2(n)\rfloor$ we can recover all digits of $i$ written as a binary number. Thus $S_{n,1}$ is separating.  Since the number of $\Sym_n$-orbits on $V$ is $n+1$, we obtain $2^{\ga-1}<n+1\leq 2^{\ga}$. Then the inequalities  $2^{|S_{n,1}|-1}\leq n<2^{|S_{n,1}|}$ imply  
$|S_{n,1}| = \ga$ and we obtain that $S_{n,1}$ contains the least possible number of elements for a separating set by part~1 of Theorem~\ref{theo1}.
\end{proof}

The next remark shows that the natural generalization of Lemma~\ref{lem_F2_m1} to the case of an arbitrary finite field does not hold. 

\begin{remark}\label{remark_F2_ex}%
\noindent
\begin{enumerate}
\item[1.] Assume that $\F=\F_3$, $m=1$ and denote 
$$M_{1,n}=\{s_1, s_2, s_{3r}\,|\, 1\leq r\leq n/3\}.$$ 
Using straightforward calculations, we can see that for $2\leq n\leq 14$ the set $M_{1,n}$ is a minimal separating set for  $\F[V]^{\Sym_n}$. On the other hand, $M_{1,n}$ contains the least possible number of elements for a separating set  (i.e., $|M_{1,n}|=\ga$) if and only if $n\leq 8$.

\item[2.] Assume that $\F=\F_4$, $m=1$, and $n=3$. Then $\ga=3$ and, therefore, $\{s_1,s_2,s_3\}$ is a minimal separating set for $\F[V]^{\Sym_3}$ by part~1 of Theorem~\ref{theo1}. 
\end{enumerate}
\end{remark}

\begin{lem}\label{lemma_reduction} Consider subsets $J\subset\{1,\ldots,n\}$ and   $A\subset \N^m$  such that $S_{J}=\{s_j\,|\,j\in J\}$ is a separating set for $\F[V]^{\Sym_n}$ and $\{\si_t(\un{\al})\,|\,1\leq t\leq n,\; \un{\al}\in A\}$  is a separating set for $\F[V^m]^{\Sym_n}$. Then $\{\si_j(\un{\al})\,|\,j\in J,\; \un{\al}\in A\}$  
is a separating set for $\F[V^m]^{\Sym_n}$.
\end{lem}
\begin{proof}
Let  $\un{u}=(u(1),\ldots,u(m)),\,\un{v}=(v(1),\ldots,v(m))\in V^m$ be separated by  $\F[V^m]^{\Sym_n}$. Then they are separated by $\si_t(\un{\al})$ for some $1\leq t\leq n$ and $\un{\al}\in A$. The vectors $u',v'\in V$ defined by $u_i'=u(1)_i^{\al_1}\cdots u(m)_i^{\al_m}$ and $v_i'=v(1)_i^{\al_1}\cdots v(m)_i^{\al_m}$ for all $1\leq i\leq n$  satisfy the following equalities: 
\begin{eq}\label{eq_lemma_reduction}
s_l(u')=\si_l(\un{\al})(\un{u})\;\;\text{and}\;\;s_l(v')=\si_l(\un{\al})(\un{v})
\end{eq}%
for all $1\leq l\leq n$. Hence $s_t(u')$ is not equal to  $s_t(v')$. Since $S_{J}$ is a separating set for $\F[V]^{\Sym_n}$, there exists $j\in J$ such that $s_j$ separates $u',v'$. Finally, equalities~\Ref{eq_lemma_reduction} imply that $\si_j(\un{\al})$ separates $\un{u}$ and $\un{v}$.  
\end{proof}

To define some set of the representatives of $\Sym_n$-orbits on $V^m$ in case $\F=\F_2$ we introduce the following notations.  
Given $r,k\geq1$ and $\un{\tau}\in\N^l$ with $l\geq rk$ we define the {\it $r$-th $k$-interval}: 
$$\int{k}{r} =  \{(r-1)k+1, \ldots, rk\},$$
\noindent{}and the {\it $\un{\tau}$-sum ranges over $\int{k}{r}$}:
$$\sumint{\un{\tau}}{k}{r} = \tau_{(r-1)k+1}+ \cdots+\tau_{rk} = 
\sum\limits_{s\in \int{k}{r}} \tau_{s}.$$
Note that for all $1\leq j\leq m$ and $\tau\in\N^{2^m}$ we have
$$\begin{array}{ccccccccc}
\{1,\ldots, 2^m\} & = &\int{2^j}{1} & \sqcup &\int{2^j}{2} & \sqcup & \cdots & \sqcup & \int{2^j}{2^{m-j}},\\
|\un{\tau}|&=&\sumint{\un{\tau}}{2^j}{1}& +& \sumint{\un{\tau}}{2^j}{2}& + &\cdots &+ &\sumint{\un{\tau}}{2^j}{2^{m-j}}.\\
\end{array}
$$
Given $s\geq0$ and $a\in\F$, denote by $a^s$ the vector $(a,\ldots,a)\in \F^s$. Then in case $\F=\F_2$ it is easy to see that as a set of representatives of $\Sym_n$-orbits on $V^m$ we can take the set  
$${\mathcal O}_{n,m}=\{e_{\un{\tau}}=(e_{\un{\tau}}(1),\ldots,e_{\un{\tau}}(m))\;|\;\un{\tau}\in \N^{2^m} \;\text{and}\;  |\un{\tau}|=n\},$$
where 
$$
\begin{array}{ccl}
e_{\un{\tau}}(1)&=&(1^{\sumint{\un{\tau}}{2^{m-1}}{1}},0^{\sumint{\,\un{\tau}}{2^{m-1}}{2}}), \\
e_{\un{\tau}}(2)&=&(1^{\sumint{\un{\tau}}{2^{m-2}}{1}},0^{\,\sumint{\un{\tau}}{2^{m-2}}{2}}, 1^{\sumint{\un{\tau}}{2^{m-2}}{3}},0^{\,\sumint{\un{\tau}}{2^{m-2}}{4}}), \\
&\vdots&\\
e_{\un{\tau}}(j)&=&(1^{\sumint{\un{\tau}}{2^{m-j}}{1}},0^{\,\sumint{\un{\tau}}{2^{m-j}}{2}}, \ldots, 1^{\sumint{\un{\tau}}{2^{m-j}}{2^j-1}},0^{\,\sumint{\un{\tau}}{2^{m-j}}{2^j}}),\\
&\vdots&\\
e_{\un{\tau}}(m)&=&(1^{\sumint{\un{\tau}}{1}{1}},0^{\,\sumint{\un{\tau}}{1}{2}}, \ldots, 1^{\sumint{\un{\tau}}{1}{2^m-1}},0^{\,\sumint{\un{\tau}}{1}{2^m}})\\
\end{array}
$$
for $1\leq j\leq m$.

\begin{example}\label{ex1} Assume $\F=\F_2$. 
\begin{enumerate}
\item[(a)] If $m=2$, then the set of all $e_{\un{\tau}}=(e_{\un{\tau}}(1),e_{\un{\tau}}(2))$ of $V^2$  with $\un{\tau}\in\N^4$ and  $|\un{\tau}|=n$ is a set of representatives of $\Sym_n$-orbits on $V^2$, where
$$
\begin{array}{ccl}
e_{\un{\tau}}(1)&=&(1^{\tau_1+\tau_2}\;\;,0^{\,\tau_3+\tau_4}\;), \\
e_{\un{\tau}}(2)&=&(1^{\tau_1},0^{\,\tau_2}, 1^{\tau_3},0^{\,\tau_4}). \\
\end{array}
$$

\item[(b)] If $m=3$, then the set of all $e_{\un{\tau}}=(e_{\un{\tau}}(1),e_{\un{\tau}}(2),e_{\un{\tau}}(3))$ of $V^3$  with $\un{\tau}\in\N^8$ and  $|\un{\tau}|=n$ is a set of representatives of $\Sym_n$-orbits on $V^3$, where
$$
\begin{array}{ccl}
e_{\un{\tau}}(1)&=&(1^{\tau_1+\tau_2+\tau_3+\tau_4}\;\;\;\;\;,0^{\,\tau_5+\tau_6+\tau_7+\tau_8}\;\;\;\;\,), \\
e_{\un{\tau}}(2)&=&(1^{\tau_1+\tau_2}\;\;,0^{\,\tau_3+\tau_4}\;, 1^{\tau_5+\tau_6}\;\;,0^{\,\tau_7+\tau_8}\;\,), \\
e_{\un{\tau}}(3)&=&(1^{\tau_1},0^{\,\tau_2}, 1^{\tau_3},0^{\,\tau_4}, 1^{\tau_5},0^{\,\tau_6}, 1^{\tau_7},0^{\,\tau_8}).\\
\end{array}
$$
\end{enumerate}
\end{example}
\medskip

Given $\un{w}\in V^m$ and $1\leq j\leq m$, we write ${\rm Del}_{j}(\un{w})$ for the vector of $V^{m-1}$ obtained from $\un{w}$ by deleting the $j$-th vector, i.e., ${\rm Del}_{j}(\un{w}) = (w_1,\ldots,w_{j-1},w_{j+1},\ldots,w_m)$. We use the next lemma to prove Theorem~\ref{theo_F2} (see below).

\begin{lem}\label{lemma_thF2} Let $\F=\F_2$ and $m\geq2$. For  $\un{\tau},\un{\theta}\in\N^{2^m}$ with  $|\un{\tau}|=|\un{\theta}|=n$ denote $\De_i=\tau_i-\theta_i$ for all $1\leq i\leq 2^m$. Then the following three conditions are equivalent:
\begin{enumerate}
\item[(A)]  for every  $1\leq j\leq m$, $0\leq r\leq \lfloor  \log_2(n)\rfloor$, and $\un{\al}\in\N^m$ with $\al_j=0$ we have $\si_{2^r}(\un{\al})(e_{\un{\tau}})=\si_{2^r}(\un{\al})(e_{\un{\theta}})$; 

\item[($A'$)] for every  $1\leq j\leq m$ the vectors ${\rm Del}_{j}(e_{\un{\tau}})$ and ${\rm Del}_{j}(e_{\un{\theta}})$  belong to the same $\Sym_n$-orbit on $V^{m-1}$;  

\item[(B)] for all $1< i\leq 2^m$ we have $\De_i=(-1)^{\xi(i-1)} \De_1$, where $\xi$ was defined in Remark~\ref{remark_F2}.
\end{enumerate}
\end{lem}
\begin{proof}
\noindent{\underline{$({\rm A})\Longleftrightarrow ({\rm A'})$}}. 
Given $\un{\al}\in\N^m$ with $\al_j=0$, denote by $\un{\be}\in\N^{m-1}$ the vector obtained from $\un{\al}$ by deleting the $j$-th coordinate. Hence
\begin{eq}\label{eq1_lemma_thF2}
\si_{t}(\un{\al})(e_{\un{\tau}}) = \si_{t}(\un{\be})({\rm Del}_{j}(e_{\un{\tau}}))
\end{eq}%
for all $1\leq t\leq n$. Therefore, Lemmas~\ref{lem_F2_m1} and~\ref{lemma_reduction} imply that conditions (A) and ($\rm A'$) are equivalent.
\medskip

Consider the next condition:
\smallskip
\begin{enumerate}
\item[(C)] for every $1\leq j\leq m$ we have 
$\tau_i + \tau_{i+2^{m-j}}=\theta_i + \theta_{i+2^{m-j}}$
for all $i\in \mathcal{I}_j$, where $\mathcal{I}_j=\int{2^{m-j}}{1}\,\sqcup\, \int{2^{m-j}}{3}\, \sqcup\, \cdots\,  \sqcup \, \int{2^{m-j}}{2^j -1}$.
\end{enumerate}%
\smallskip

\noindent{}We will show that ($\rm A'$) is equivalent to (C) and (B) is equivalent to (C).

\medskip
\noindent{\underline{$({\rm A'})\Longleftrightarrow ({\rm C})$}}. At first, let us show that in the partial cases of $j=1,2$ it can be easily shown that conditions  ($\rm A'$) and (C) are equivalent. Namely, if $j=1$, then the condition that ${\rm Del}_{1}(e_{\un{\tau}})$ and ${\rm Del}_{1}(e_{\un{\theta}})$  belong to the same $\Sym_n$-orbit on $V^{m-1}$ is equivalent to   
$$\tau_{i}+\tau_{i+2^{m-1}} = \theta_{i}+\theta_{i+2^{m-1}}$$ 
for all $i\in\{1,\ldots,2^{m-1}\}=\int{2^{m-1}}{1}=\mathcal{I}_j$. Similarly, in the case of $j=2$ condition ($\rm A'$) is equivalent to 
$$\tau_{i}+\tau_{i+2^{m-2}} = \theta_{i}+\theta_{i+2^{m-2}}$$ 
for all $i\in \int{2^{m-2}}{1}\sqcup \int{2^{m-2}}{3}=\mathcal{I}_j$.

To consider the general case we define the tableau $\T$ with $m$ rows and $2^m$ columns such that for $1\leq j\leq m$ the $j$-th row is 
$$(1^{2^{m-j}},0^{2^{m-j}}, \ldots, 1^{2^{m-j}},0^{2^{m-j}}).$$%

\noindent{}In other words, 


$$\T=\left(
\begin{array}{c}
\underbrace{1 \;\cdot\cdot\cdot\cdot\cdot\cdot\cdot\cdot\; 1}_{2^{m-1}} \;  \underbrace{0 \;\cdot\cdot\cdot\cdot\cdot\cdot\cdot\cdot\; 0}_{2^{m-1}} \\
\underbrace{1 \cdots 1}_{2^{m-2}} \; \underbrace{0 \cdots 0}_{2^{m-2}} \;  \underbrace{1 \cdots 1}_{2^{m-2}} \; \underbrace{0 \cdots 0}_{2^{m-2}}  \\
\vdots \\
1 \; 0 \cdot\cdot\cdot\cdot\cdot\cdot 1 \; 0 \; 
1 \; 0 \cdot\cdot\cdot\cdot\cdot\cdot 1 \; 0 \\
\end{array}
\right)$$

\noindent{}It is easy to see that for $1\leq i\leq 2^m$ the $i$-th column $\T_i$ of $\T$ is $(2^m-i)$ written as a binary number, i.e., 
\begin{eq}\label{eq_lemma_thF2}
\T_i=
\left(
\begin{array}{c}
\xi_{m-1}(2^m-i)\\
\xi_{m-2}(2^m-i)\\
\vdots\\
\xi_{0}(2^m-i) \\
\end{array}
\right).
\end{eq}

We remove the $j$-th row from $\T$ and denote the resulting tableau by $\T^{(j)}$. We write $\T^{(j)}_1,\ldots,\T^{(j)}_{2^m}$ for the columns of $\T^{(j)}$. Note that ${\rm Del}_{j}(e_{\un{\tau}})$ and ${\rm Del}_{j}(e_{\un{\theta}})$ belong to the same $\Sym_n$-orbit on $V^{m-1}$ if and only if 
\begin{eq}\label{eq0_lemma_thF2}
\begin{array}{c}
\tau_{i_1}+\cdots+\tau_{i_s} = \theta_{i_1}+\cdots+\theta_{i_s} 
\text{ for each sequence } \\
1\leq i_1<\cdots<i_s\leq 2^m 
\text{ with }\T^{(j)}_{i_1}=\cdots=\T^{(j)}_{i_s}. \\
\end{array}%
\end{eq}

Since columns of $\T$ are pairwise different, for every column $\T^{(j)}_i$ there exists a unique column $\T^{(j)}_l$ with $\T^{(j)}_i=\T^{(j)}_l$ and $i\neq l$. We will describe all such pairs $(i,l)$ with $i<l$. By equality~\Ref{eq_lemma_thF2}, we have $\T^{(j)}_i=\T^{(j)}_l$ for some $i<l$ if and only if the following conditions hold:
\begin{enumerate}
\item[(a)] $\xi_{m-k}(2^m-i)=\xi_{m-k}(2^m-l)$ for every $k\in\{1,\ldots,m\}\backslash\{j\}$, 

\item[(b)] $\xi_{m-j}(2^m-i)=1$, 

\item[(c)] $\xi_{m-j}(2^m-l)=0$.
\end{enumerate}
Obviously, $\xi_{m-k}(2^{m-j})=0$ for every $k\in\{1,\ldots,m\}\backslash\{j\}$ and $\xi_{m-j}(2^{m-j})=1$. Hence conditions (a)--(c) imply that $(2^m-i) - 2^{m-j} = 2^m - l$ and 
\begin{equation}\label{eq2_lemma_thF2}
l=i+2^{m-j}.
\end{equation}%
Therefore, condition~\Ref{eq0_lemma_thF2} is equivalent to 
\begin{eq}\label{eq3_lemma_thF2}
\begin{array}{c}
\tau_i +\tau_{i+2^{m-j}} = \theta_i +\theta_{i+2^{m-j}} 
\text{ for each } 1\leq i\leq 2^m \text{ with } i+2^{m-j}\leq 2^m \\
\text{ and satisfying condition (b)}. 
\end{array}
\end{eq}

Considering the $j$-th row of $\T$ we obtain the  next two properties for the set $\mathcal{I}_j$:
\begin{enumerate}
\item[$\bullet$]  For $1\leq i\leq 2^m$  we have $\xi_{m-j}(2^m-i)=1$ if and only if $i\in \mathcal{I}_j$.

\item[$\bullet$] The set $\{1,\ldots,2^m\}$ is the union without intersections of sets $\{i,i+2^{m-j}\}$, where $i$ ranges over $\mathcal{I}_j$. 
\end{enumerate}
Thus, condition~\Ref{eq3_lemma_thF2} is equivalent to
$$
\begin{array}{c}
\tau_i +\tau_{i+2^{m-j}} = \theta_i +\theta_{i+2^{m-j}} 
\text{ for all }i\in \mathcal{I}_j. 
\end{array}
$$%
The required is proven. 

\medskip
\noindent{\underline{$({\rm B})\Longleftrightarrow ({\rm C})$}}. We can rewrite condition (C) as

\smallskip
\begin{enumerate}
\item[(${\rm C'}$)] for all $1\leq j\leq m$ and $i\in \mathcal{I}_j$ we have 
$\De_{i+2^{m-j}}=-\De_i$.
\end{enumerate}
\smallskip

Assume that condition (B) holds. Then for all $i,j$ as in condition (${\rm C'}$) we have $\De_i= (-1)^{\xi(i-1)}\De_1$ and $\De_{i+2^{m-j}}= (-1)^{\xi(i-1+2^{m-j})}\De_1$. Since  $(2^m-i)+(i-1)=2^m-1=(1\ldots1)_2$, then for every $1\leq k\leq m$ we have 
$$
\xi_{m-k}(2^m-i)=\xi_{m-k}(i-1)+1\;\;\text{in}\;\; \F_2.
$$%
\noindent{}This equality together with the first property of the set $\mathcal{I}_j$ implies that 
\begin{eq}
\xi_{m-j}(i-1)=0 \text{ if and only if }
i\in \mathcal{I}_j.
\end{eq}
Then $\xi(i-1+2^{m-j})=\xi(i-1) + 1$ in $\F_2$. Therefore,  condition (${\rm C'}$) is valid. 

Assume that condition (${\rm C'}$) holds. To establish condition (B) we will show by induction on $l$ that   
\begin{eq}\label{eq3}
\De_{2^l+1}=(-1)^{\xi(2^l)} \De_1,\; 
\De_{2^l+2}=(-1)^{\xi(2^l+1)} \De_1,\;\ldots\; ,\De_{2^{l+1}}=(-1)^{\xi(2^{l+1}-1)} \De_1
\end{eq}%
for all $0\leq l<m$. If $l=0$, then~\Ref{eq3} is $\De_2=-\De_1$ and it follows from condition ($\rm C'$) with $j=m$ and $i=1$.

Assume that for some $0<l_0<m$ equalities~\Ref{eq3} are valid for all $0\leq l<l_0$, i.e., condition (B) is valid for all $1<i\leq 2^{l_0}$. For  $j=m-l_0$ and $i\in \int{2^{l_0}}{1}=\{1,\ldots,2^{l_0}\}\subset \mathcal{I}_j$ equality~$\De_{i+2^{m-j}} = -\De_{i}$ holds. Applying  condition (B) to it we obtain that    
$$\De_{i+2^{l_0}}=-(-1)^{\xi(i-1)}\De_1 = (-1)^{\xi(i+2^{l_0}-1)}\De_1.$$
Thus~\Ref{eq3} is valid for $l=l_0$ and condition (B) is proven.

\end{proof}

\begin{thm}\label{theo_F2} 
If $\F=\F_2$, $m\geq1$, and $n\geq 2$, then the following set is a minimal   
separating set for $\F[V^m]^{\Sym_n}$:
$$S_{n,m} = \{\si_{2^r}(\un{\al})\,|\,r\geq0,\; |\un{\al}|\geq 1,\; r+|\un{\al}|-1\leq \lfloor \log_2(n)\rfloor,\; \al_j\in\{0,1\} \text{ for }1\leq j\leq m\}.$$
\end{thm}
\begin{proof} {\bf 1.} At first we show that $S_{n,m}$ is a separating set. 

The proof is by induction on $m\geq1$. If $m=1$, then the required follows from Lemma~\ref{lem_F2_m1}. Assume that the statement of the theorem holds for $m-1$.

Assume that $\un{u},\un{v}\in V^m$ are not separated by all invariants from $S_{n,m}$. To complete the proof it is enough to show that $\un{u},\un{v}$ are not separated by $\F[V^m]^{\Sym_n}$. Obviously, without loss of generality we can assume that $\un{u},\un{v}$ are elements of ${\mathcal O}_{n,m}$, i.e., $\un{u}=e_{\un{\tau}}$ and $\un{v}=e_{\un{\theta}}$ for some $\un{\tau},\un{\theta}\in\N^{2^m}$ with  $|\un{\tau}|=|\un{\theta}|=n$. Since $e_{\un{\tau}}$ and $e_{\un{\theta}}$ are not separated if and only if $\un{\tau}=\un{\theta}$, we have to show that $\un{\tau}=\un{\theta}$.

As in Lemma~\ref{lemma_thF2}, denote $\De_i=\tau_i-\theta_i$ for all $1\leq i\leq 2^m$. For every  $1\leq j\leq m$ we apply the induction hypothesis to obtain that condition (A) from Lemma~\ref{lemma_thF2} holds. 
Hence condition (B) from Lemma~\ref{lemma_thF2} is also valid. Condition (B) implies that  
$$2^{m}|\De_1| = \sum_{i=1}^{2^m} |\De_i|\leq \sum_{i=1}^{2^m} (\tau_i + \theta_i )= |\un{\tau}|+|\un{\theta}|=2n.$$
Therefore,
\begin{equation}\label{eq4}
2^{m-1}|\De_1|\leq n.
  \end{equation}

If $\De_1=0$, then condition (B) implies that $\un{\tau}=\un{\theta}$, i.e.,  $e_{\un{\tau}}$ and $e_{\un{\theta}}$ are not separated by $\F[V^m]^{\Sym_n}$.

Assume that $|\De_1|>0$. For every $0\leq r\leq \lfloor \log_2(|\De_1|)\rfloor$   inequality~\Ref{eq4} implies that $r\leq \lfloor \log_2(n)\rfloor  - m + 1$. Therefore, for $\un{\al}=(1,\ldots,1)\in\N^m$ the invariant $\si_{2^r}(\un{\al})$ belongs to $S_{n,m}$. We have that $\si_{2^r}(\un{\al})(e_{\un{\tau}}) = \si_{2^r}(1)(e_{\tau_1})=s_{2^r}(e_{\tau_1})$ and similar equalities hold for $e_{\un{\theta}}$. Since $e_{\un{\tau}}$ and $e_{\un{\theta}}$ are not separated by $\si_{2^r}(\un{\al})$, it follows from part (b) of Lemma~\ref{lem_F2_binary} that $\tau_1=\theta_1$; a contradiction.

Since we obtained that $e_{\un{\tau}}$ and $e_{\un{\theta}}$ are not separated by $\F[V^m]^{\Sym_n}$, we established that $S_{n,m}$ is separating.

\bigskip
\noindent{\bf 2.} Let us show that $S_{n,m}$ is a minimal separating set, where $m\geq2$ by Lemma~\ref{lem_F2_m1}. Assume the contrary. Then  we can consider an invariant $f=\si_{2^r}(\un{\al})\in S_{n,m}$ with minimal $|\un{\al}|$ such that $S_{n,m}\backslash\{f\}$ is a separating set for $\F[V^m]^{\Sym_n}$. Obviously, without loss of generality we can assume that $m=|\un{\al}|$ and $\un{\al}=(1,\ldots,1)\in\N^m$. Hence, $0\leq r\leq \lfloor \log_2(n)\rfloor - m+1$. 

For $a=n-2^{r+m-1}\geq0$ define $\un{\tau},\un{\theta}\in\N^{2^m}$ as follows:
$$
\tau_i=\left\{
\begin{array}{cl}
2^r + a, & \text{ if } i=4 \\
2^r, & \text{ if } \xi(i-1) \text{ is even and } i\neq 4 \\
0, & \text{ if } \xi(i-1) \text{ is odd}\\
\end{array}
\right.,
$$
$$
\theta_i=\left\{
\begin{array}{cl}
a, & \text{ if } i=4 \\
0, & \text{ if } \xi(i-1) \text{ is even and } i\neq 4 \\
2^r, & \text{ if } \xi(i-1) \text{ is odd}\\
\end{array}
\right.
$$
for all $1\leq i\leq 2^m$. Thus for every $1\leq i\leq 2^m$ we have that $\De_i=\tau_i-\theta_i$ satisfies $\De_i=(-1)^{\xi(i-1)}\De_1$ and $\De_1= 2^r$. 

Denote by $A_0(m)$ the number of integers $0\leq i<2^m$, where $\xi(i)$ is even. Similarly, denote by $A_1(m)$ the number of integers $0\leq i<2^m$, where $\xi(i)$ is odd. Since for an even integer $0\leq i<2^m$ parities of $\xi(i)$ and $\xi(i+1)$ are distinct, we obtain that 
$$A_0(m)=A_1(m)=2^{m-1}.$$
Then $|\un{\tau}|= a + 2^r A_0(m) = n$ and $|\un{\theta}|= a + 2^r A_1(m) = n$. Therefore, $\un{\tau},\un{\theta}$ satisfy condition (B) of Lemma~\ref{lemma_thF2} as well as condition (A). Hence, for each $h\in S_{n,m}$ different from $\si_{2^t}(\un{\al})$ for every $0\leq t\leq \lfloor \log_2(n)\rfloor - m+1$ we have $h(e_{\un{\tau}})=h(e_{\un{\theta}})$. On the other hand, for every $0\leq t\leq \lfloor \log_2(n)\rfloor - m+1$ we have
$$\si_{2^t}(\un{\al})(e_{\un{\tau}})= s_{2^t}(e_{\tau_1}) = s_{2^t}(e_{2^r}) = \xi_t(2^r) = 
\left\{
\begin{array}{cl}
1,& \text{ if }t=r \\
0,& \text{ otherwise } \\
\end{array}
\right.,$$
$$\si_{2^t}(\un{\al})(e_{\un{\theta}})= s_{2^t}(e_{\theta_1}) = s_{2^t}(e_0) = 0.$$
In particular, $\si_{2^t}(\un{\al})(e_{\un{\tau}})=\si_{2^t}(\un{\al})(e_{\un{\theta}})$
for all  $0\leq t\leq \lfloor \log_2(n)\rfloor - m+1$ with $t\neq r$, but $\si_{2^r}(\un{\al})(e_{\un{\tau}})\neq \si_{2^r}(\un{\al})(e_{\un{\theta}})$.
Therefore,  $S_{n,m}\backslash\{f\}$ is not a separating set; a contradiction.
\end{proof}

\begin{example}\label{ex2} For $\F=\F_2$ consider the minimal  
separating set $S_{n,m}$ from Theorem~\ref{theo_F2}:
\begin{enumerate}
\item[(a)] if $m=2$, then $S_{n,2}$ is  
$$\si_{2^r}(1,0),\; \si_{2^r}(0,1),\; \text{ where }\; 0\leq r\leq \lfloor \log_2(n)\rfloor;$$
$$\si_{2^r}(1,1),\; \text{ where }\; 0\leq r\leq \lfloor \log_2(n)\rfloor - 1.$$

\item[(b)] if $m=3$, then $S_{n,3}$ is  
$$\si_{2^r}(1,0,0),\; \si_{2^r}(0,1,0),\;\si_{2^r}(0,0,1),\; \text{ where }\; 0\leq r\leq \lfloor \log_2(n)\rfloor;$$
$$\si_{2^r}(1,1,0),\;\si_{2^r}(1,0,1),\;\si_{2^r}(0,1,1),\; \text{ where }\; 0\leq r\leq \lfloor \log_2(n)\rfloor - 1;$$
$$\si_{2^r}(1,1,1), \text{ where }\; 0\leq r\leq \lfloor \log_2(n)\rfloor - 2.$$
\end{enumerate}
\end{example}

To illustrate the proofs of Theorem~\ref{theo_F2} and key Lemma~\ref{lemma_thF2}, we repeat these proofs in the partial case of $m=2$ in Example~\ref{ex_proof_m2} and in the partial case of $m=3$ in Example~\ref{ex_proof_m3}.

\begin{example}\label{ex_proof_m2}
Let $\F=\F_2$ and $m=2$. We will show that $S_{n,2}$ from Theorem~\ref{theo_F2} is a minimal separating set.

\medskip
\noindent{\bf 1.} At first we show that $S_{n,2}$ is a separating set. Assume that $\un{u}=e_{\un{\tau}}$ and $\un{v}=e_{\un{\theta}}$ are not separated by $S_{n,2}$ for some  $\un{\tau},\un{\theta}\in\N^{4}$ with $|\un{\tau}|=|\un{\theta}|=n$. Using Lemma~\ref{lem_F2_m1} we assume that ${\rm Del}_{j}(e_{\un{\tau}})$ and ${\rm Del}_{j}(e_{\un{\theta}})$ belong to the same $\Sym_n$-orbit for  $j=1,2$.
Since ${\rm Del}_{1}(e_{\un{\tau}})=(1^{\tau_1},0^{\,\tau_2}, 1^{\tau_3},0^{\,\tau_4})\in V$ 
and  ${\rm Del}_{1}(e_{\un{\theta}})$ has a similar form, we obtain that  $\tau_{i}+\tau_{i+2} = \theta_{i}+\theta_{i+2}$  for all $i\in\int{2}{1}=\{1,2\}$. Since ${\rm Del}_{2}(e_{\un{\tau}})=(1^{\tau_1+\tau_2},0^{\,\tau_3+\tau_4})\in V$
and  ${\rm Del}_{2}(e_{\un{\theta}})$ has a similar form, we obtain that  $\tau_{i}+\tau_{i+1} = \theta_{i}+\theta_{i+1}$ for all $i\in \int{1}{1}\sqcup \int{1}{3}=\{1,3\}$. Therefore, by straightforward calculations we verified that condition (C) from Lemma~\ref{lemma_thF2} holds for $\un{\tau},\un{\theta}$. Let us remark that the tableau $\T$ from the proof of Lemma~\ref{lemma_thF2} is  
$$\T=\left(
\begin{array}{c}
1 \; 1 \; 0 \; 0 \\
1 \; 0 \; 1 \; 0\\
\end{array}
\right).$$
Obviously the obtained linear equations on entries of $\un{\tau}$ and $\un{\theta}$ imply that
$$\De_1=-\De_2=-\De_3=\De_4,$$
where $\De_i=\tau_i-\theta_i$ for all $1\leq i\leq 4$, i.e.,  condition (B) of Lemma~\ref{lemma_thF2} holds for $m=2$. We complete the proof of the fact that $\un{\tau}=\un{\theta}$ as in the proof of  Theorem~\ref{theo_F2}.

\medskip
\noindent{\bf 2.}  Let us show that $S_{n,2}$ is a minimal separating set. Assume the contrary. Then  we can consider an invariant $f=\si_{2^r}(\un{\al})\in S_{n,2}$ such that $S_{n,2}\backslash\{f\}$ is a separating set for $\F[V^2]^{\Sym_n}$. Since $S_{n,1}$ is a minimal separating set for $\F[V]^{\Sym_n}$ by Lemma~\ref{lem_F2_m1},   we have $\un{\al}=(1,1)$. For $a=n-2^{r+1}\geq0$ define $\un{\tau}=(2^r,0,0,2^r+a)$ and $\un{\theta}=(0,2^r,2^r,a)$. Then $|\un{\tau}|=|\un{\theta}|=n$ and we obtain a contradiction as in the proof of Theorem~\ref{theo_F2}.
\end{example}

\begin{example}\label{ex_proof_m3}
Let $\F=\F_2$ and $m=3$. We will show that $S_{n,3}$ from Theorem~\ref{theo_F2} is a minimal separating set. 

\medskip
\noindent{\bf 1.} Assume that $\un{u}=e_{\un{\tau}}$ and $\un{v}=e_{\un{\theta}}$ are not separated by $S_{n,3}$ for some  $\un{\tau},\un{\theta}\in\N^{8}$ with $|\un{\tau}|=|\un{\theta}|=n$. Using Example~\ref{ex_proof_m2} we assume that ${\rm Del}_{j}(e_{\un{\tau}})$ and ${\rm Del}_{j}(e_{\un{\theta}})$ belong to the same $\Sym_n$-orbit for all $j\in\{1,2,3\}$.
Since ${\rm Del}_{1}(e_{\un{\tau}})$ is 
$$
\left(
\begin{array}{c}
(1^{\tau_1+\tau_2}\;\;,0^{\,\tau_3+\tau_4}\;, 1^{\tau_5+\tau_6}\;\;,0^{\,\tau_7+\tau_8}\;\,) \\
(1^{\tau_1},0^{\,\tau_2}, 1^{\tau_3},0^{\,\tau_4}, 1^{\tau_5},0^{\,\tau_6}, 1^{\tau_7},0^{\,\tau_8})\\
\end{array}
\right)\in V^2,
$$
and  ${\rm Del}_{1}(e_{\un{\theta}})$ has a similar form, we obtain that  $\tau_{i}+\tau_{i+4} = \theta_{i}+\theta_{i+4}$  for all $i\in\int{4}{1}=\{1,\ldots,4\}$. Since ${\rm Del}_{2}(e_{\un{\tau}})$ is 
$$
\left(\begin{array}{c}
(1^{\tau_1+\tau_2+\tau_3+\tau_4}\;\;\;\;\;,0^{\,\tau_5+\tau_6+\tau_7+\tau_8}\;\;\;\;\,)\\
(1^{\tau_1},0^{\,\tau_2}, 1^{\tau_3},0^{\,\tau_4}, 1^{\tau_5},0^{\,\tau_6}, 1^{\tau_7},0^{\,\tau_8})\\
\end{array}
\right)\in V^2
$$
and  ${\rm Del}_{2}(e_{\un{\theta}})$ has a similar form, we obtain that $\tau_1+\tau_3=\theta_1+\theta_3$, $\tau_2+\tau_4=\theta_2+\theta_4$, $\tau_5+\tau_7=\theta_5+\theta_7$, $\tau_6+\tau_8=\theta_6+\theta_8$. In other words, $\tau_{i}+\tau_{i+2} = \theta_{i}+\theta_{i+2}$ for all $i\in \int{2}{1}\sqcup \int{2}{3}=\{1,2,5,6\}$. Since ${\rm Del}_{3}(e_{\un{\tau}})$ is 
$$
\left(\begin{array}{c}
(1^{\tau_1+\tau_2+\tau_3+\tau_4}\;\;\;\;\;,0^{\,\tau_5+\tau_6+\tau_7+\tau_8}\;\;\;\;\,)\\
(1^{\tau_1+\tau_2}\;\;,0^{\,\tau_3+\tau_4}\;, 1^{\tau_5+\tau_6}\;\;,0^{\,\tau_7+\tau_8}\;\,) \\
\end{array}
\right)\in V^2
$$
and  ${\rm Del}_{3}(e_{\un{\theta}})$ has a similar form, we obtain that $\tau_{i}+\tau_{i+1} = \theta_{i}+\theta_{i+1}$
for all $i\in \int{1}{1}\,\sqcup\, \int{1}{3}\,\sqcup \, \int{1}{5} \,\sqcup \, \int{1}{7} = \{1,3,5,7\}$. Therefore, by straightforward calculations we verified that condition (C) from Lemma~\ref{lemma_thF2} holds for $\un{\tau},\un{\theta}$. Let us remark that the tableau $\T$ from the proof of Lemma~\ref{lemma_thF2} is  
$$\T=\left(
\begin{array}{c}
1 \; 1 \; 1 \; 1\; 0 \; 0 \; 0 \; 0\\
1 \; 1 \; 0 \; 0\; 1 \; 1 \; 0 \; 0\\
1 \; 0 \; 1 \; 0\; 1 \; 0 \; 1 \; 0\\
\end{array}
\right).$$
It is not difficult to see that the obtained linear equations on entries of $\un{\tau}$ and $\un{\theta}$ imply that
$$\De_1=-\De_2=-\De_3=\De_4=-\De_5=\De_6=\De_7=-\De_8,$$
i.e., condition (B) of Lemma~\ref{lemma_thF2} holds for $m=3$. We complete the proof of the fact that $\un{\tau}=\un{\theta}$ as in the proof of  Theorem~\ref{theo_F2}.

\medskip
\noindent{\bf 2.}  Let us show that $S_{n,3}$ is a minimal separating set. Assume the contrary. Then  we can consider an invariant $f=\si_{2^r}(\un{\al})\in S_{n,3}$ such that $S_{n,3}\backslash\{f\}$ is a separating set for $\F[V^2]^{\Sym_n}$. Since $S_{n,1}$ ($S_{n,2}$, respectively) is a minimal separating sets in case $m=1$ ($m=2$, respectively),   we have $\un{\al}=(1,1,1)$. For $a=n-2^{r+2}\geq0$ define $\un{\tau}=(2^r,0,0,2^r+a,0,2^r,2^r,0)$ and $\un{\theta}=(0,2^r,2^r,a,2^r,0,0,2^r)$. Then $|\un{\tau}|=|\un{\theta}|=n$ and we obtain a contradiction as in the proof of Theorem~\ref{theo_F2}.
\end{example}

Given $m_0<m$, the notion of \emph{expansion} of a set $S\subset\F[V^{m_0}]^{\Sym_n}$  to a set  $S^{[m]}$ of $\F[V^m]^{\Sym_n}$ can be found for example in~\cite{Lopatin_Reimers_1}. Denote by $\sisep(n)$ the minimal number $m_0$ such that the expansion of some separating set $S$ for $\F[V^{m_0}]^{\Sym_n}$ produces a separating set for $\F[V^m]^{\Sym_n}$ for all $m \geq m_0$. In~\cite{Lopatin_Reimers_1} it was proven that $\sigma(n) \leq \lfloor \frac{n}{2} \rfloor + 1$ over an arbitrary field $\F$.

\begin{cor}\label{cor} 
For $\F=\F_2$ we have that
\begin{enumerate}
\item[(a)] $\besep(\F[V^m]^{\Sym_n})=2^{\lfloor  \log_2(n)\rfloor}$; 

\item[(b)] $\sisep(n)=\lfloor  \log_2(n)\rfloor + 1$.  
\end{enumerate}
\end{cor}
\begin{proof}

\noindent{\bf (a)} We start with the following claim:
\begin{eq}\label{claim1}
\begin{array}{c}
\text{If }s_t \text{ separates } e_a \text{ and } e_b \text{ for some } 1\leq t\leq n \text{ and }0\leq a,b\leq n, \\
\text{ then there exists } r\geq0 \text{  with } 2^r\leq t \text{ such that } s_{2^r} \text{ separates } e_a \text{ and }e_b.
\end{array}
\end{eq}

To prove this claim we use notation $e_a^{(n)}$ for $e_a$ and $s_t^{(n)}$ for $s_t$ to highlight $n$. Note that for all $1\leq l\leq n_0\leq n$ and $0\leq i\leq n$ we have $s_l^{(n)}(e_i^{(n)})=s_l^{(n_0)}(e_i^{(n_0)})=\binom{i}{l}$. Therefore, $s_t^{(t)}(e_a^{(t)})\neq s_t^{(t)}(e_b^{(t)})$. Taking $t$ as $n$ in Lemma~\ref{lem_F2_m1} we obtain that there exists $r\geq0$ satisfying $2^r\leq t$ and 
$s_{2^r}^{(t)}(e_a^{(t)})\neq s_{2^r}^{(t)}(e_b^{(t)})$. Thus $s_{2^r}^{(n)}$ separates $e_a^{(n)}$ and $e_b^{(n)}$. The claim is proven.

Note that the maximal degree of elements of the separating set $S_{n,m}$ is equal to $2^{\lfloor  \log_2(n)\rfloor}$. Therefore  $\besep(\F[V^m]^{\Sym_n})\leq 2^{\lfloor  \log_2(n)\rfloor}$. Assume that $\besep(\F[V^m]^{\Sym_n})<2^{\lfloor  \log_2(n)\rfloor}$. Then $\besep(\F[V]^{\Sym_n})<2^{\lfloor  \log_2(n)\rfloor}$ and Claim~\Ref{claim1} implied that there exists a proper separating subset of $S_{n,1}$; a contradiction to Lemma~\ref{lem_F2_m1}.

\medskip
\noindent{\bf (b)} 
By Theorem~\ref{theo_F2} and~\cite[Remark~1.1]{Lopatin_Reimers_1} (cf.~\cite[Remark~1.3]{domokos2007}) we have that $\sisep(n)$ is equal to the maximal $k$ such that for $\un{\al}=(1,\ldots,1)$ ($k$ times) we have $\si_{2^r}(\un{\al})\in S_{n,m}$ for some $r,m$. Part (b) is proven.
\end{proof}

\bigskip\bigskip\bigskip


\bibliographystyle{siam}
\bibliography{main}

\end{document}